\documentclass{amsart}
\usepackage{amssymb,amsmath}
\usepackage{amsthm}
\usepackage{mathrsfs}
\usepackage{enumerate, indentfirst}
\usepackage[fleqn,tbtags]{mathtools}

\newtheorem{theorem}{Theorem}[section]
\newtheorem{lemma}[theorem]{Lemma}

\theoremstyle{definition}

 \theoremstyle{remark}
\newtheorem{remark}[theorem]{Remark}

\numberwithin{equation}{section}

\renewcommand{\phi}{\varphi}

\newcommand{\lmu}{\mathscr{L}^2(\mu)}
\newcommand{\lnu}{\mathscr{L}^2(\nu)}

\newcommand{\alg}{\mathscr{R}}

\newcommand{\dupN}{\mathbb{N}}
\newcommand{\seq}[1]{(#1_{n})_{n\in\dupN}}

\newcommand{\nen}{n\in\mathbb{N}}
\newcommand{\dupR}{\mathbb{R}}

\newcommand{\hil}{\mathfrak{H}}

\newcommand{\kil}{\mathfrak{K}}

\DeclarePairedDelimiterX\sipn[2]{(}{)_{\nu}}{#1\,\delimsize\vert\,#2}
\DeclarePairedDelimiterX\sipm[2]{(}{)_{\mu}}{#1\,\delimsize\vert\,#2}
\DeclarePairedDelimiterX\Sipn[2]{(}{)_{\nu}}{#1\,\delimsize\vert\,#2}
\DeclarePairedDelimiterX\Sipm[2]{(}{)_{\mu}}{#1\,\delimsize\vert\,#2}
\newcommand{\cha}[1]{\chi_{_{#1}}}

\begin{document}
\title{Lebesgue decomposition via Riesz orthogonal decomposition}


\author{Zsigmond Tarcsay}
\address{Zs. Tarcsay, Department of Applied Analysis, E\"otv\"os L. University, P\'azm\'any P\'eter s\'et\'any 1/c., Budapest H-1117, Hungary; }
\email{tarcsay@cs.elte.hu}

\begin{abstract}
We give a simple and short proof of the classical  Lebesgue de\-comp\-os\-ition theorem of measures via the  Riesz orthogonal decomposition theorem of Hilbert spaces. The tools we employ are elementary Hilbert space techniques.
\end{abstract}


\maketitle

\section{Introduction}

J. von Neumann in \cite{Neumann} gave a very elegant proof of the classical Radon--Nikodym differentiation theorem, namely, he proved that the Radon--Nikodym theorem follows (relatively easily) from the Riesz representation theorem for bounded linear functionals. Our purpose in this paper is to show how the Lebesgue de\-comp\-osition theorem derives from  Riesz' orthogonal decomposition theorem. More precisely, if $\mu$ and $\nu$ are finite measures on a fixed measurable space $(T,\alg)$ then  the $\mu$-absolute continuous and $\mu$-singular parts of $\nu$ correspond to an appropriate orthogonal decomposition $\mathfrak{M}\oplus\mathfrak{M}^{\perp}$ of $\lnu$.

A very similar approach can be used also by discussing several general Lebesgue-type decomposition problems such as decomposing finitely additive set functions \cite{realanalex}, positive operators on Hilbert spaces \cite{tarcsay}, nonnegative Hermitian forms \cite{sebestytarcsaytitkos}, and representable functionals on $^*$-algebras \cite{tarcsayrepr}. The key in the mentioned papers just as in the present note  is the Riesz orthogonal decomposition theorem applied to a suitable  subspace of an appropriate $\mathscr{L}^2$-space.

We must also mention that  Neumann's proof simultaneously proves Lebesgue's decomposition, at least after making minimal modifications, see e.g. Rudin \cite{Rudin}. His treatment is undoubtedly more elegant and simpler than ours. Our only claim is nothing but to point out the deep connection between Lebesgue decomposition and orthogonal decomposition.

Our proof of the Lebesgue decomposition theorem is based on the following easy lemma which states that the so called multivalued part of a closed linear relation is closed itself, cf. \cite{Arens} or \cite{Hassi2007}.  The proof is just an easy exercise, however, we present it here for the sake of completeness.
%%%%%%%%%%%%%%%%%%%%%%%%%%%%%%%%%%%%%%
\begin{lemma}\label{M(L) lemma}
Let $\hil$ and $\kil$ be (real or complex) Hilbert spaces and let $L$ be a linear subspace of $\hil\times\kil$, that is to say,  a linear relation from $\hil$ into $\kil$. Then
\begin{align}\label{M(L)}
\mathfrak{M}(L):=\big\{k\in\kil:~ \exists(h_n,k_n)_{\nen} ~from~L~with~ h_n\rightarrow0,k_n\rightarrow k\big\}
\end{align}
is a closed subspace of $\kil$.
\end{lemma}
\begin{proof}
Note first that for any $k\in\kil$ the assertion $k\in\mathfrak{M}(L)$ is equivalent to $(0,k)\in \overline{L}$, where $\overline{L}$ denotes the closure of $L$ in the Cartesian product $\hil\times\kil$. Now, if $(k_n)_{\nen}$ is any sequence from $\mathfrak{M}(L)$ with limit point $k\in\kil$, then clearly, $(0,k_n)_{\nen}$ converges in $\overline{L}$, namely to $(0,k)$. Consequently, $(0,k)\in\overline{L}$ according to the closedness of $\overline{L}$ and hence $k\in\mathfrak{M}(L)$.
\end{proof}
\section{The Lebesgue decomposition theorem}

Henceforth, we fix two finite measures $\mu$ and $\nu$ on a measurable space $(T,\alg)$, where $T$ is a non-empty set, and $\alg$ is a $\sigma$-algebra of subsets of $T$. For $E\in\alg$, the characteristic function of $E$ is denoted by $\cha{E}$. The vector space of $\alg$-simple functions, i.e., the linear span of characteristic functions  is denoted by $\mathscr{S}$. We also associate  (real) Hilbert spaces $\lmu$ and $\lnu$ to the measures $\mu$ and $\nu$, respectively, which are endowed with the usual inner products, denoted by $\sipm{\cdot}{\cdot}$ and $\sipn{\cdot}{\cdot}$, respectively. Note that functions belonging to $\lmu$ (resp., to $\lnu$) are just $\mu$-almost everywhere (resp., $\nu$-almost everywhere) determined. For any measurable function $f:T\rightarrow\dupR$ and $c\in\dupR$ we define the measurable set $[f\leq c]$ by letting
\begin{equation*}
    [f\leq c]:=\{x\in T: f(x)\leq c\}.
\end{equation*}
There are defined $[f\geq c]$, $[f=c]$, $[f\neq c]$, etc., similarly.

Recall the notions of absolute continuity and singularity: $\nu$ is said to be \emph{absolutely continuous} with respect to $\mu$ (shortly, $\mu$-absolute continuous) if $\mu(E)=0$ implies $\nu(E)=0$ for all $E\in\alg$; $\nu$ is called \emph{singular} with respect to $\mu$ (shortly, $\mu$-singular) if there exists $S\in\alg$ such that $\mu(S)=0$ and $\nu(T\setminus S)=0$. The Lebesgue decomposition theorem states that  $\nu$ admits a unique decomposition $\nu=\nu_a+\nu_s$, where $\nu_a$ is $\mu$-absolute continuous and $\nu_s$ is $\mu$-singular. The uniqueness can be easily proved via a simple measure theoretic argument (see e.g. \cite{Halmos} or \cite{Rudin}). The essential part of this statement is in the existence of the decomposition.

Let us consider now the following linear subspace of $\lnu$:
\begin{equation*}%\label{M}
    \mathfrak{M}=\left\{f\in\lnu : \exists\seq{\phi}\subset\mathscr{S},~\phi_n\rightarrow0~\textrm{in~}\lmu,~ \phi_n\rightarrow f\textrm{~in~}\lnu\right\}.
\end{equation*}
Then clearly, $\mathfrak{M}=\mathfrak{M}(L)$ by choosing
\begin{equation*}
L:=\{(\phi,\phi):\phi\in\mathscr{S}\}\subseteq\lmu\times\lnu.
\end{equation*}
 Consequently, $\mathfrak{M}$ is closed, according to Lemma \ref{M(L) lemma}. Let us consider the orthogonal projection $P$ of $\lnu$ onto $\mathfrak{M}$, the existence of which is guaranteed by the classical Riesz orthogonal decomposition theorem. Let us introduce the following two (signed) measures:
\begin{equation}\label{decomposition}
    \nu_a(E):=\int\limits_E (I-P)1~d\nu,\qquad \nu_s(E):=\int\limits_E P1~d\nu,\qquad E\in\alg.
\end{equation}
Clearly, $\nu=\nu_a+\nu_s$. We state that this is  the Lebesgue decomposition of $\nu$ with respect to $\mu$:
\begin{theorem}\label{Lebesgue}
Let $\mu$ and $\nu$ be finite measures on a measurable space $(T,\alg)$. Then both $\nu_a$ and $\nu_s$ from \eqref{decomposition} are (finite) measures such that $\nu_a$ is $\mu$-absolute continuous, and $\nu_s$ is $\mu$-singular.
\end{theorem}
\begin{proof}
We start by proving the $\mu$-absolute continuity of $\nu_a$: let $E\in\alg$ be any measurable set with $\mu(E)=0$. Then $\cha{E}\in\mathfrak{M}$ (choose $\phi_n:=\cha{E}$ for all integer $n$), and therefore
\begin{equation*}
    \nu_a(E)=\int\limits_{E} (I-P)1~d\nu=\sipn{\cha{E}}{(I-P)1}=\sipn{(I-P)\cha{E}}{1}=\sipn{0}{1}=0.
\end{equation*}
This means that the signed measure $\nu_a$ is absolute continuous with respect to $\mu$. In order to prove the $\mu$-singularity of $\nu_s$, let us consider a sequence $\seq{\phi}$ from $\mathscr{S}$ with $\phi_n\rightarrow0$ in $\lmu$, and $\phi_n\rightarrow P1$ in $\lnu$.  By turning to an appropriate subsequence along the classical Riesz argument \cite{Riesz}, we may also assume that $\phi_n\rightarrow 0$ $\mu$-a.e. This means that $P1=0$ $\mu$-a.e., and therefore that $\nu_s$ and $\mu$ are singular with respect to each other:
\begin{equation}\label{P1=0}
    \mu([P1\neq0])=0\qquad\textrm{and}\qquad\nu_s([P1=0])=\int\limits_{[P1=0]} P1~d\nu=0.
\end{equation}
It remains only to show that $\nu_a$ and $\nu_s$ are \emph{positive} measures, i.e., $0\leq P1\leq 1$ $\nu$-a.e. Indeed, the left side of \eqref{P1=0} yields $\mu([P1<0])=0$ and hence $\cha{[P1<0]}\in\mathfrak{M}$. That gives
\begin{equation*}
    0\geq\int\limits_{[P1<0]} P1~d\nu=\Sipn{P1}{\cha{[P1<0]}}=\Sipn{1}{\cha{[P1<0]}}=\nu([P1<0])\geq0,
\end{equation*}
which means that $P1\geq0$ $\nu$-a.e. A very same argument shows that $P1\leq1$ $\nu$-a.e.
\end{proof}
\begin{remark}
Observe that $P1$ is, in fact, the characteristic function of the set $[P1\neq0]$: since $0\leq P1\leq1$, it follows that $0\leq P1-(P1)^2$, and we also have
\begin{equation*}
    \int\limits_T P1-(P1)^2~d\nu=\int\limits_T P1\cdot(I-P)1~d\nu=\sipn{P1}{(I-P)1}=0,
\end{equation*}
which yields $P1=\cha{[P1\neq0]}$, indeed. Consequently, by letting $S:=[P1\neq0]$ the standard form of the Lebesgue decomposition is obtained as follows:
\begin{equation*}
    \nu_a(E)=\nu(E\setminus S),\quad\nu_s(E)=\nu(E\cap S),\qquad E\in\alg,
\end{equation*}
c.f. \cite[32 \S,  Theorem C]{Halmos}.
\end{remark}

\end{document}